\documentclass[11pt,twoside]{amsart}
\usepackage{latexsym,amssymb,amsmath}
\numberwithin{equation}{section}
\newtheorem{theorem}{Theorem}[section]
\newtheorem{lemma}[theorem]{Lemma}
\newtheorem{proposition}[theorem]{Proposition}
\newtheorem{corollary}[theorem]{Corollary}

\theoremstyle{definition}
\newtheorem{definition}[theorem]{Definition} 
\newtheorem{procedure}[theorem]{Procedure} 
\newtheorem{remark}[theorem]{Remark}
\newtheorem{example}[theorem]{Example}

\begin{document}


\newcommand{\m}[1]{\marginpar{\addtolength{\baselineskip}{-3pt}{\footnotesize
\it #1}}}
\newcommand{\A}{\mathcal{A}}
\newcommand{\K}{\mathcal{K}}
\newcommand{\knd}{\mathcal{K}^{[d]}_n}
\newcommand{\F}{\mathcal{F}}
\newcommand{\N}{\mathbb{N}}
\newcommand{\pr}{\mathbb{P}}
\newcommand{\Z}{\mathbb{Z}}
\newcommand{\R}{\mathbb{R}}
\newcommand{\I}{\mathit{I}}
\newcommand{\G}{\mathcal{G}}
\newcommand{\D}{\mathcal{D}}
\newcommand{\x}{\mathbf{x}}
\newcommand{\lcm}{\operatorname{lcm}}
\newcommand{\ndp}{N_{d,p}}
\newcommand{\tor}{\operatorname{Tor}}
\newcommand{\reg}{\operatorname{reg}}
\newcommand{\mf}{\mathfrak{m}}

\def\bb{{{\rm \bf b}}}
\def\cc{{{\rm \bf c}}}


\title[]{Associated primes of powers of edge ideals}
\author{Jos\'e Mart\'{\i}nez-Bernal}
\address{
Departamento de
Matem\'aticas\\
Centro de Investigaci\'on y de Estudios Avanzados del
IPN\\
Apartado Postal
14--740 \\
07000 Mexico City, D.F. } \email{jmb@math.cinvestav.mx}
\thanks{The first and third authors were partially supported by SNI}

\author{Susan Morey}
\address{Department of Mathematics \\
Texas State University\\
601 University Drive\\ 
San Marcos, TX 78666.}
\email{morey@txstate.edu}

\author{Rafael H. Villarreal}
\address{
Departamento de
Matem\'aticas\\
Centro de Investigaci\'on y de Estudios
Avanzados del
IPN\\
Apartado Postal
14--740 \\
07000 Mexico City, D.F.
}
\email{vila@math.cinvestav.mx}


\begin{abstract} Let $G$ be a graph and let $I$ be its
 edge ideal. Our main result shows that the sets of associated primes
 of the powers of 
 $I$ form an ascending chain. It is known that the sets of associated
 primes of $I^i$ and $\overline{I^i}$ stabilize for large $i$. We
 show that their corresponding stable sets are equal.
To show 
 our main result we use a classical result of Berge from matching
 theory and certain notions from combinatorial optimization. 
\end{abstract}
\maketitle
\section{Introduction}

Let $G$ be a simple {\it graph\/} with finite vertex set 
$X=\{x_1,\ldots,x_n\}$, i.e., $G$ is the set $X$ together with a family of 
subsets of $X$ of cardinality $2$, called edges, none 
of which is included in another. The sets of
vertices and edges of $G$ are denoted by $V(G)$
and $E(G)$ respectively. We shall always assume that $G$ has no
isolated vertices, i.e., every vertex of $G$ has to occur in at least
one edge. Let $R=K[x_1,\ldots,x_n]$ be a polynomial ring 
over a field $K$. The {\it edge ideal\/} of $G$, 
denoted by $I=I(G)$, is the ideal of $R$
generated by all square-free monomials $x_ix_j$ such that
$\{x_i,x_j\}\in E(G)$. 
The assignment $G\mapsto
I(G)$ gives a natural one to one
correspondence between the family of graphs and the family 
of monomial ideals generated by square-free monomials of degree $2$.

Let $G$ be a graph and let $I=I(G)$ be its edge ideal. In this paper
we will examine the sets of 
associated primes of the powers of $I$, that is, the sets
$${\rm Ass}(R/I^k)=\{{\mathfrak p}\subset R \, | \, {\mathfrak p} \
{\mbox{\rm is prime and }} 
{\mathfrak p}=(I^k:c)\ 
{\mbox{\rm for some }} c\in R\},\ \ k\geq 1.$$ 
Since $I$ is a monomial ideal of a polynomial ring $R$, the
associated primes 
will be monomial primes, which are primes that are 
generated by subsets of the variables, see
\cite[Proposition~5.1.3]{monalg}. 
The associated primes of $I$
correspond to 
minimal vertex covers of the graph $G$ and ${\rm Min}(R/I)={\rm
Ass}(R/I)$,  where ${\rm Min}(R/I)$ denotes the
set of minimal primes of $I$, see \cite{Vi2}. For edge ideals, ${\rm
Ass}(R/I)\subset {\rm Ass}(R/I^k)$ for all positive integers $k$. In
the case where 
equality holds for all $k$, the ideal $I$ is said to 
be {\em normally torsion-free}.

In \cite{brod}, Brodmann showed that the sets ${\rm Ass}(R/I^k)$ stabilize
for large $k$. That is, there exists a positive integer $N_1$ such that
${\rm Ass}(R/I^k)={\rm Ass}(R/I^{N_1})$ for all $k \geq N_1$. A minimal
such $N_1$ is called the {\em index of stability} of $I$. One important
result in this area establishes that $N_1=1$ if and only if $G$ is a
bipartite graph \cite[Theorem~5.9]{ITG}. A useful upper bound for
$N_1$ was 
shown in \cite[Corollary~4.3]{AJ}, namely that if $G$ is a connected
non-bipartite graph with $n$ vertices, $s$ leaves, and the smallest odd
cycle of $G$ has length $2k+1$, then  $N_1 \leq n-k-s$. We make use of
this upper bound in Example~\ref{intcl1.m2}.

Although the sets ${\rm Ass}(R/I^k)$ are known to stabilize for large
$k$, their behavior for small $k$ can be erratic. Finding the stable
set ${\rm Ass}(R/I^{N_1})$ is complicated by the fact that a prime
ideal $\mathfrak{p}$ that is associated to a low power of an ideal
$I$ need not be associated 
to higher powers. For example, \cite[Example,~p.~2]{McAdam} gives an example,
due to A. Sathaye, of
an ideal $I$ in a ring $R$ and a prime $\mathfrak{p}$ for which
$\mathfrak{p} \in {\rm Ass}(R/I^k)$ for $k$ 
even and $\mathfrak{p} \not\in {\rm Ass}(R/I^k)$ for $k$ odd for all
$k$ below a stated bound. When,
for an ideal $I$, $\mathfrak{p} \in {\rm Ass}(R/I^k)$ implies
$\mathfrak{p}\in {\rm 
  Ass}(R/I^{k+1})$ for all $k \geq 1$, one says that the sets ${\rm
  Ass}(R/I^k)$ form an {\it ascending chain\/}. Although this property is
highly desirable, few classes of ideals are known to possess it. Examples
of monomial ideals for which the sets ${\rm Ass}(R/I^k)$ do not form
ascending chains can be found in \cite[Section~4]{HH} (stated in terms
of depths), and \cite[Example~4.18]{edge-ideals}.

Let $\overline{I^k}$ denote the integral closure of $I^k$. An ideal
$I$ is called {\it normal} if $I^k=\overline{I^k}$ for all
$k\geq 1$. By results of
Ratliff \cite{ratliff,ratliff-increasing}, one has that the sets ${\rm
Ass}(R/{\overline{I^k}})$ form an ascending chain which stabilizes for
large $k$. Thus, there exists $N_2$ such that ${\rm Ass}(R/\overline{I^k})={\rm
Ass}(R/\overline{I^{N_2}})$ for $k\geq N_2$. The set ${\rm
Ass}(R/\overline{I^{N_2}})$ is nicely described in \cite{mcadam}, and
for edge ideals of graphs the set  ${\rm Ass}(R/I^{N_1})$ is
described in \cite{AJ}.   

Our main result is:

\noindent {\bf Theorem~\ref{persistence-edge-ideals}}{\it\ If $I$ is
  the edge ideal of a graph, then ${\rm 
Ass}(R/I^k) \subset{\rm Ass}(R/I^{k+1})$ for all $k$. That 
is, the sets of associated primes of the powers of $I$ form an
ascending chain.
}

There are two cases where the sets of associated primes of a
square-free monomial ideal are known to form an ascending chain. The
first case is the family of normal ideals (as was pointed out above),
which includes, for instance, ideals of vertex covers of perfect graphs
\cite{persistence-dm,FHV,perfect}. The second case
is the family of graphs with at  
least one leaf \cite{edge-ideals}, which is now a particular case of
our main result.   

In a more general setting, i.e., when $I\neq(0)$ is an ideal of a
commutative 
Noetherian domain, Ratliff showed that $(I^{k+1}\colon I)=I^k$ for all
large $k$ \cite[Corollary~4.2]{ratliff} and that equality holds for
all $k$ when $I$ is normal \cite[Proposition~4.7]{ratliff}. We show
that equality holds for 
all $k$ when $I$ is an edge ideal. 

\noindent {\bf Lemma~\ref{jan29-2011}}{\it\ $(I^{k+1}\colon I)=I^k$
  for $k\geq 1$.} 

This lemma is central to the proof of our main result. 
To show this lemma, we need to link the algebraic and
combinatorial data. This is achieved using matching theory 
and basic notions from combinatorial optimization. 

Given an edge $f$, we denote by $G^f$ the
graph obtained from $G$ by duplicating the two vertices of $f$ (see
Definition~\ref{parallelization-def}). The {\it deficiency} of $G$,
denoted by ${\rm def}(G)$, is the number of vertices left uncovered
by any maximum matching of $G$. The matching number of $G$ is denoted
by $\nu(G)$ (see Definition~\ref{matching-number}). Using a formula of Berge (see
Theorem~\ref{berge-formula}), we compare the deficiencies of $G$ and $G^f$. 

Our main combinatorial result is:

\noindent {\bf Theorem~\ref{pepe-vila-berge}}{\it\ ${\rm
def}(G^f)=\delta$ for all $f\in E(G)$ 
if and only if ${\rm
def}(G)=\delta$ and $\nu(G^f)=\nu(G)+1$ for all $f\in E(G)$.}

As a byproduct, we present the following
characterization of graphs with a perfect matching. 

\noindent {\bf Corollary~\ref{pepe-vila}}{\it\ $G$ has a perfect matching if 
and only if $G^f$ has a perfect matching for every edge $f$ of $G$.}

In general, for ideals in commutative Noetherian rings, 
${\rm Ass}(R/\overline{I^{N_2}})$ is a subset of
${\rm Ass}(R/I^{N_1})$, see \cite[Proposition~3.17]{McAdam}. We show that
for edge ideals these stable sets are equal. 

\noindent {\bf Theorem~\ref{ass=assic}}{\it\ ${\rm Ass}(R/I^k)={\rm
Ass}(R/\overline{I^k})$ for $k\geq 
\max\{N_1,N_2\}$.}

As an application we show that an edge ideal $I$ is normally
torsion-free if and only if  
${\rm Ass}(R/\overline{I^i})={\rm Ass}(R/I)$ for $i\geq 1$ (see
Corollary~\ref{ntf-ass-corollary}). 

Throughout the paper we introduce most of the 
notions that are relevant for our purposes. For unexplained
terminology we refer to \cite{diestel,Mats,bookthree}. Two excellent
references for the 
general theory of asymptotic prime divisors in commutative Noetherian
rings are \cite{huneke-swanson-book} and \cite{McAdam}.

\section{Perfect matchings and persistence of associated primes}

In this section we give a characterization of graphs with a perfect
matching and show that the sets of associated primes of powers of an
edge ideal form an ascending chain. We continue using the
definitions and terms from the introduction.

Let $G$ be a graph with vertex set $X=\{x_1,\ldots,x_n\}$ and let 
$I=I(G)\subset R$ be its edge ideal. In what follows
$F=\{f_1,\ldots,f_q\}$
denotes the set of all monomials $x_ix_j$ such that $\{x_i,x_j\}\in
E(G)$. As usual, we use 
$x^a$ as an 
abbreviation for $x_1^{a_1} \cdots x_n^{a_n}$ and we set
$|a|=a_1+\cdots+a_n$, where $a=(a_i)\in \mathbb{N}^n$. For convenience
we will consider $0$ to be an element of $\mathbb{N}$. We also use
$f^c$ as an abbreviation 
for $f_1^{c_1}\cdots f_q^{c_q}$, where $c=(c_i)\in\mathbb{N}^q$.

\begin{definition}\label{parallelization-def}\rm Following Schrijver
\cite{Schr2}, the {\it
duplication\/} of a vertex $x_i$  
of a graph $G$ means extending its vertex set $X$ by a
new vertex $x_i'$ 
and replacing
$E(G)$ by
$$    
E(G)\cup\{(e\setminus\{x_i\})\cup\{x_i'\}\vert\, x_i\in e\in
E(G)\}.
$$
The {\it deletion\/} of $x_i$, denoted by
$G\setminus\{x_i\}$,  is the graph formed from
$G$ by deleting the vertex $x_i$ and all edges containing
$x_i$. A graph obtained from $G$ by a sequence of
deletions and 
duplications of vertices is called a {\it parallelization\/} of $G$. 
\end{definition} 

It is not difficult to verify that these two operations commute.
If $a=(a_i)$ is a vector in $\mathbb{N}^n$, we denote by $G^a$
the graph obtained from
$G$ by successively deleting any vertex $x_i$ with $a_i=0$ and
duplicating $a_i-1$ times any vertex $x_i$ if $a_i\geq 1$ 
(cf. \cite[p.~53]{golumbic}). The notion of a parallelization was used
in \cite{cm-mfmc,symboli} to describe the symbolic Rees algebra of an
edge ideal. This notion has its origin in combinatorial optimization
and has been used to describe the max-flow min-cut property of
clutters \cite{covers,Schr2}.

\begin{example}\label{april9-09} Let $G$ be the graph of Fig. 1 
and let $a=(3,3)$. We set $x_i^1=x_i$ for $i=1,2$. The parallelization 
$G^a$ is a complete bipartite graph 
with bipartition $V_1=\{x_1^1,x_1^2,x_1^3\}$ and
$V_2=\{x_2^1,x_2^2,x_2^3\}$. Note that $x_i^k$ is a vertex, i.e., $k$
is an index not an exponent.
\vspace{0.8cm}
$$
\begin{array}{cccc}
\setlength{\unitlength}{.04cm}
\thicklines
\begin{picture}(80,35)
\put(0,10){\circle*{3.1}}
\put(0,40){\circle*{3.1}}
\put(-15,42){$x_1$}
\put(-15,3){$x_2$}
\put(0,10){\line(0,1){30}}
\put(10,15){$G$}
\put(-20,-15){\mbox{Fig. 1. Graph}}
\end{picture}
&
\setlength{\unitlength}{.04cm}
\thicklines
\begin{picture}(80,35)
\put(30,40){\circle*{3.1}}
\put(60,40){\circle*{3.1}}
\put(0,40){\circle*{3.1}}
\put(-20,-15){\mbox{Fig. 2. Duplications of $x_1$}}
\put(0,10){\circle*{3.1}}
\put(-15,42){$x_1^1$}
\put(18,42){$x_1^2$}
\put(45,42){$x_1^3$}
\put(-15,3){$x_2^1$}
\put(0,10){\line(0,1){30}}
\put(0,10){\line(1,1){30}}
\put(0,10){\line(2,1){60}}
\put(35,15){$G^{(3,1)}$}
\end{picture}
&
\ \ \ \ \ \ \ \ \ \ 
&
\setlength{\unitlength}{.04cm}
\thicklines
\begin{picture}(80,35)
\put(0,10){\circle*{3.1}}
\put(-15,42){$x_1^1$}
\put(18,42){$x_1^2$}
\put(45,42){$x_1^3$}
\put(-15,3){$x_2^1$}
\put(18,3){$x_2^2$}
\put(45,3){$x_2^3$}
\put(0,10){\line(0,1){30}}
\put(0,10){\line(1,1){30}}
\put(0,10){\line(2,1){60}}
\put(30,10){\circle*{3.1}}
\put(30,10){\line(0,1){30}}
\put(30,10){\line(1,1){30}}
\put(30,10){\line(-1,1){30}}
\put(60,10){\circle*{3.1}}
\put(60,10){\line(0,1){30}}
\put(60,10){\line(-2,1){60}}
\put(60,10){\line(-1,1){30}}
\put(0,40){\circle*{3.1}}
\put(30,40){\circle*{3.1}}
\put(60,40){\circle*{3.1}}
\put(70,15){$G^{(3,3)}$}
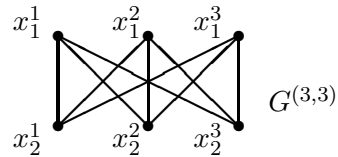
\put(-25,-15){\mbox{Fig. 3. Duplications of $x_1$ and $x_2$}}
\end{picture}
\end{array}
$$
\end{example}

\bigskip

\begin{definition}\label{matching-number} Two edges of $G$ are {\it
independent} if they do not intersect. A {\it matching\/} of $G$ is a
set of pairwise 
independent edges. The {\it matching number\/} of
$G$, denoted by $\nu(G)$, is the size of any maximum matching of $G$.
A matching that covers all the vertices of 
$V(G)$  is called a {\it perfect matching} of $G$. 
\end{definition}

A very readable and comprehensive reference about matchings in finite
graphs is the book of Lov\'asz and Plummer \cite{matching-theory}. 

Given a graph $G$, the {\it edge-subring\/} of $G$ is the 
subring $K[G]=K[x_ix_j\vert\, \{x_i,x_j\}\in E(G)]$.

\begin{lemma}\label{multiset-perfect-matching} Let $G$ be a graph
with vertex set 
$X=\{x_1,\ldots,x_n\}$ and let $a=(a_1,\ldots,a_n)\in\mathbb{N}^n$.
Then $G^a$ has a perfect matching if and only if $x^a\in K[G]$. 
\end{lemma}

\begin{proof} We may assume that $a_i\geq 1$ for all $i$, because if
$a$ has zero entries we can use the induced subgraph on the vertex 
set $\{x_i\vert\, a_i>0\}$. The vertex set of $G^a$ is
$$
X^a=\{x_1^1,\ldots,x_1^{a_1},\ldots,x_i^1,\ldots,x_i^{a_i},
\ldots,x_n^1,\ldots,x_n^{a_n}\}
$$
and the edges of $G^a$ are exactly those pairs of the form 
$\{x_{i}^{k_{i}},x_{j}^{k_{j}}\}$ with $i\neq j$, $k_{i}\leq
a_{i}$, $k_j\leq a_j$, for some edge $\{x_i,x_j\}$ of $G$. We can
regard $x^a$ as an ordered multiset 
$$
x^a=x_1^{a_1}\cdots x_n^{a_n}=(\underbrace{x_1\cdots x_1}_{a_1})\cdots
(\underbrace{x_n\cdots x_n}_{a_n})
$$ 
on the set $X$, that is, we can identify the monomial $x^a$ with the multiset  
$$X_a=\{\underbrace{x_1,\ldots,x_1}_{a_1},\ldots,
\underbrace{x_n,\ldots,x_n}_{a_n}\}
$$ 
on $X$ in which each variable is uniquely identified with an integer
between $1$ and $|a|$. This 
integer is the position, from left to right, of $x_i$ in $X_a$. There
is a bijective map 
$$
\begin{array}{ccccccccccc}
1&2&\cdots&a_1&a_1+1&\cdots&a_1+a_2&\cdots&a_1+\cdots+a_{n-1}+1&\cdots&
a_1+\cdots+a_n\\
\downarrow&\downarrow&\cdots&\downarrow&\downarrow&\cdots&\downarrow
&\cdots&\downarrow&\cdots&\downarrow\\
x_1&x_1&\cdots&x_1&x_2&\cdots&x_2&\cdots&x_n&\cdots&x_n\\
\downarrow&\downarrow&\cdots&\downarrow&\downarrow&\cdots&\downarrow
&\cdots&\downarrow&\cdots&\downarrow\\
x_1^1&x_1^2&\cdots&x_1^{a_1}&x_2^1&\cdots&x_2^{a_2}&\cdots&x_n^1&\cdots&x_n^{a_n}.\\
\end{array}
$$
Hence if $G^a$ has a perfect matching, then the perfect matching
induces a 
factorization of $x^a$ in which each factor corresponds to an
edge of $G$, i.e., $x^a\in K[G]$. Conversely, if $x^a\in K[G]$
we can factor $x^a$ as a product of monomials corresponding to edges
of $G$ and this factorization induces a perfect matching of $G^a$. 
\end{proof}

Note that the process of passing from the vertex set $X$ to the
set $X^a$ and the multiset $X_a$ used in the lemma above can also be
used to view a 
general monomial as a square-free monomial in a polynomial ring with
additional variables. This is referred to as {\em  polarization} in
the literature. The copies of $x_i$ that are used are called shadows
of $x_i$. Conversely, a square-free monomial $M$ in the ring $K[X^a]$
can be viewed as a monomial in the ring $K[X]$ by setting the
exponent of $x_i$ to be the number of shadows of $x_i$ that divide
$M$. This process is called {\em depolarization}.     
 
Given an edge $f=\{x_i,x_j\}$ of a graph $G$, we denote by $G^f$ or
$G^{\{x_i,x_j\}}$ the graph obtained
from $G$ by successively duplicating the vertices $x_i$ and $x_j$,
i.e., $G^f:=G^{\mathbf{1}+e_i+e_j}$, where $e_i$ is the $i${\it th} unit
vector in $\mathbb{R}^n$ and $\mathbf{1}=(1,\ldots,1)$. 

\begin{example} Consider the graph $G$ of Fig. 4, where vertices
are labeled with $i$ instead of $x_i$. The duplication of the
vertices $x_1$ and 
$x_2$ of $G$ is shown in Fig. 6.
$$
\begin{array}{ccccc}
\setlength{\unitlength}{.040cm} \thicklines
\begin{picture}(0,50)(120,20)
\put(0,0){\circle*{4.2}} \put(60,0){\circle*{4.2}} \put(0,30){\circle*{4.2}}
\put(30,60){\circle*{4.2}} \put(60,30){\circle*{4.2}}\put(30,30){\circle*{4.2}}
\put(30,15){\circle*{4.2}}

\put(0,0){\line(1,0){60}}\put(0,0){\line(0,1){30}}\put(0,0){\line(2,1){30}}
\put(60,0){\line(0,1){30}}
\put(0,30){\line(1,1){30}}\put(60,30){\line(-1,1){30}}\put(30,15){\line(0,1){15}}
\put(30,30){\line(0,1){30}}\put(60,0){\line(-2,1){30}}

\newcommand{\lb}[1]{\tiny $#1$}
\put(-6,0){\lb{4}} \put(64,0){\lb{3}} \put(-6,28){\lb{5}} \put(29,63){\lb{1}}
\put(64,28){\lb{2}}\put(24,28){\lb{6}} \put(24,15){\lb{7}}
\put(8,-20){Fig. 4.
$G$}
\end{picture}
& \ \ &
\setlength{\unitlength}{.04cm} \thicklines
\begin{picture}(0,50)(30,20)
\put(0,0){\circle*{4.2}}
\put(60,0){\circle*{4.2}} \put(0,30){\circle*{4.2}}
\put(30,60){\circle*{4.2}} \put(60,30){\circle*{4.2}}\put(30,30){\circle*{4.2}}
\put(30,15){\circle*{4.2}}\put(20,40){\circle*{4.2}}

\put(0,0){\line(1,0){60}}\put(0,0){\line(0,1){30}}\put(0,0){\line(2,1){30}}
\put(60,0){\line(0,1){30}}
\put(0,30){\line(1,1){30}}\put(60,30){\line(-1,1){30}}\put(30,15){\line(0,1){15}}
\put(30,30){\line(0,1){30}}\put(60,0){\line(-2,1){30}}\put(20,40){\line(-2,-1){20}}
\put(20,40){\line(1,-1){10}}\put(20,40){\line(4,-1){40}}

\newcommand{\lb}[1]{\tiny $#1$}
\put(-6,0){\lb{4}} \put(64,0){\lb{3}} \put(-6,28){\lb{5}} \put(29,63){\lb{1}}
\put(64,28){\lb{2}}\put(24,28){\lb{6}} \put(24,15){\lb{7}}\put(22,42){\lb{1'}}
\put(-12,-20){Fig. 5. $G^{(2,1,1,1,1,1,1)}$}
\end{picture}

& &\ \ \setlength{\unitlength}{.04cm} \thicklines
\begin{picture}(0,50)(-60,20)
\put(0,0){\circle*{4.2}} \put(60,0){\circle*{4.2}} \put(0,30){\circle*{4.2}}
\put(30,60){\circle*{4.2}} \put(60,30){\circle*{4.2}}\put(30,30){\circle*{4.2}}
\put(30,15){\circle*{4.2}}\put(20,40){\circle*{4.2}}

\put(0,0){\line(1,0){60}}\put(0,0){\line(0,1){30}}\put(0,0){\line(2,1){30}}
\put(60,0){\line(0,1){30}}
\put(0,30){\line(1,1){30}}\put(60,30){\line(-1,1){30}}\put(30,15){\line(0,1){15}}
\put(30,30){\line(0,1){30}}\put(60,0){\line(-2,1){30}}\put(20,40){\line(-2,-1){20}}
\put(20,40){\line(1,-1){10}}\put(20,40){\line(4,-1){40}}

\newcommand{\lb}[1]{\tiny $#1$}
\put(-6,0){\lb{4}} \put(64,0){\lb{3}} \put(-6,28){\lb{5}} \put(29,63){\lb{1}}
\put(64,28){\lb{2}}\put(24,28){\lb{6}} \put(24,15){\lb{7}}\put(22,41){\lb{1'}}
\put(-21,-20){Fig. 6. $G^{(2,2,1,1,1,1,1)}=G^{\{x_1,x_2\}}$}

\put(43.5,40){\circle*{4.2}}\put(44,40){\line(-1,0){24}}\put(43,40){\line(-2,3){12}}
\put(44,40){\line(2,-5){17}}\put(33,41){\lb{2'}}
\end{picture}

\end{array}
$$
\end{example}

\vspace{1.5cm}

Recall that ${\rm def}(G)$, the {\it
deficiency\/} of $G$, is given by ${\rm def}(G)=|V(G)|-2\nu(G)$, 
where $\nu(G)$ is the matching number of $G$. Hence ${\rm def}(G)$ is
the number of vertices left uncovered by any maximum matching.

\begin{lemma}\label{feb9-2011} Let $G$ be a graph and 
let $a\in\mathbb{N}^n$ and $c\in\mathbb{N}^q$. Then 
\begin{itemize}
\item[{(a)}] $x^a=x^\delta f^c$, 
where $|\delta|={\rm def}(G^a)$ and $|c|=\nu(G^a)$.
\item[{(b)}] $x^a$ belongs to $I(G)^k\setminus
I(G)^{k+1}$ if and only if $k=\nu(G^a)$.
\item[{(c)}] $(G^a)^f=(G^a)^{\{x_i,x_j\}}$
for any edge $f=\{x_i^{k_i},x_j^{k_j}\}$ of $G^a$.
\end{itemize}
\end{lemma}

\begin{proof} Parts (a) and (b) follow using the bijective map used
in the proof of 
Lemma~\ref{multiset-perfect-matching}. To show (c) we use the
notation used in the proof of Lemma~\ref{multiset-perfect-matching}.
We now prove the inclusion $E((G^a)^f)\subset
E((G^a)^{\{x_i,x_j\}})$. Let $y_i$ 
and $y_j$ be the duplications 
of $x_i^{k_i}$ and $x_j^{k_j}$ respectively. We also denote the
duplications of $x_i$ and $x_j$ by $y_i$ and $y_j$ respectively. The common
vertex set of $(G^a)^f$ and $(G^a)^{\{x_i,x_j\}}$ is
$V(G^a)\cup\{y_i,y_j\}$. Let $e$ be an edge of $(G^a)^f$. If
$e=\{y_i,y_j\}$ or $e\cap\{y_i,y_j\}=\emptyset$, then clearly $e$ is
an edge of $(G^a)^{\{x_i,x_j\}}$. Thus, we may assume that
$e=\{y_i,x_\ell^{k_\ell}\}$. Then $\{x_i^{k_i},x_\ell^{k_\ell}\}\in
E(G^a)$, so $\{x_i,x_\ell\}\in E(G)$. Hence $\{x_i,x_\ell^{k_\ell}\}$
is in $E(G^a)$, so $e=\{y_i,x_\ell^{k_\ell}\}$ is an edge of
$(G^a)^{\{x_i,x_j\}}$. This proves the inclusion ``$\subset$''. The
other inclusion follows using similar arguments (arguing backwards). 
\end{proof}

\begin{theorem}{\rm (Berge; see
\cite[Theorem~3.1.14]{matching-theory})}\label{berge-formula} Let $G$
be a graph. Then 
$${\rm def}(G)=\max\{c_0(G\setminus S)-|S|\, \vert\, S\subset V(G)\},$$
where $c_0(G)$ denotes the number of odd components $($components
with an odd number of vertices$)$ of a graph $G$.
\end{theorem}

We come to the main combinatorial result of this section.

\begin{theorem}\label{pepe-vila-berge} Let $G$ be a graph. Then 
${\rm def}(G^f)=\delta$ for all $f\in E(G)$ if and only if ${\rm
def}(G)=\delta$ and $\nu(G^f)=\nu(G)+1$ for all $f\in E(G)$.
\end{theorem}

\begin{proof} Assume that ${\rm def}(G^f)=\delta$ for all $f\in
E(G)$. In general, ${\rm def}(G)\geq {\rm
def}(G^f)$ for any $f\in E(G)$. We proceed by contradiction. Assume
that ${\rm def}(G)>\delta$. Then, by Berge's theorem, there is an $S\subset V(G)$
such that $c_0(G\setminus S)-|S|>\delta$. We set $r=c_0(G\setminus S)$ and
$s=|S|$. Let $H_1,\ldots,H_r$ be the odd 
components of $G\setminus S$.

Case (I): $|V(H_k)|\geq 2$ for some $1\leq k\leq r$. Pick an edge
$f=\{x_i,x_j\}$ of $H_k$. Consider the
parallelization $H_k'$ obtained from $H_k$ by duplicating the vertices $x_i$
and $x_j$, i.e., $H_k'=H_k^f$. The odd connected components of
$G^f\setminus S$ are $H_1,H_2,\ldots,H_{k-1},H_k',H_{k+1}\ldots,H_r$. Thus
$$
c_0(G^f\setminus S)-|S|>\delta={\rm def}(G^f).
$$ 
This contradicts
Berge's theorem when applied to $G^f$.

Case (II): $|V(H_k)|=1$ for $1\leq k\leq r$. Notice that in this case
$S\neq\emptyset$ because $G$ has no isolated vertices. Pick
$f=\{x_i,x_j\}$ an edge of $G$ with $\{x_i\}=V(H_1)$ and $x_j\in S$. 
Let $y_i$ and $y_j$ be the duplications of $x_i$
and $x_j$ respectively. The odd components of
$G^f\setminus(S\cup\{y_j\})$ are $H_1,\ldots,H_r,\{y_i\}$. 
Thus 
$$
c_0(G^f\setminus(S\cup\{y_j\}))-|S\cup\{y_j\}|=c_0(G\setminus
S)-|S|>\delta={\rm def}(G^f).
$$ 
This again contradicts Berge's theorem when applied to $G^f$. Therefore
${\rm def}(G)={\rm def}(G^f)$ for all $f\in E(G)$. Consequently
$\nu(G^f)=\nu(G)+1$ for all $f\in E(G)$. The converse follows readily
using the definition of ${\rm def}(G)$ and ${\rm def}(G^f)$.
\end{proof}

The result of Theorem~\ref{pepe-vila-berge} depends upon the 
deficiency of $G^f$ being constant for all $f$. In general,
the deficiencies of $G$ and $G^f$ need not be equal.

\begin{example} Consider the graph $G$ of Fig. 7, where vertices
are labeled with $i$ instead of $x_i$. The duplication of the vertices $x_3$ and
$x_4$ of $G$ is shown in Fig. 8.
$$
\begin{array}{ccc}
\setlength{\unitlength}{.04cm} \thicklines
\begin{picture}(0,40)(80,-5)
\put(0,0){\circle*{4.2}} \put(30,0){\circle*{4.2}} \put(-20,20){\circle*{4.2}}
\put(-20,-20){\circle*{4.2}}
\put(50,20){\circle*{4.2}}\put(50,-20){\circle*{4.2}}

\put(0,0){\line(-1,-1){20}}\put(0,0){\line(-1,1){20}}\put(0,0){\line(1,0){30}}
\put(30,0){\line(1,1){20}} \put(30,0){\line(1,-1){20}}

\newcommand{\lb}[1]{\tiny $#1$}
\put(-27,20){\lb{1}} \put(-27,-20){\lb{2}} \put(-9,-2){\lb{3}}
\put(34,-2){\lb{4}}
\put(55,20){\lb{5}}\put(55,-20){\lb{6}}\put(-23,-40){Fig. 7.
$\mbox{def}(G)=2$}
\end{picture}

& \ \ & 
\setlength{\unitlength}{.04cm} \thicklines
\begin{picture}(0,40)(-40,-5)
\put(0,0){\circle*{4.2}} \put(30,0){\circle*{4.2}} \put(-20,20){\circle*{4.2}}
\put(-20,-20){\circle*{4.2}}
\put(50,20){\circle*{4.2}}\put(50,-20){\circle*{4.2}}
\put(0,20){\circle*{4.2}}\put(30,20){\circle*{4.2}}

\put(0,0){\line(-1,-1){20}}\put(0,0){\line(-1,1){20}}\put(0,0){\line(1,0){30}}
\put(30,0){\line(1,1){20}} \put(30,0){\line(1,-1){20}}
\put(0,20){\line(-1,0){20}}\put(0,20){\line(-1,-2){20}}
\put(30,20){\line(1,0){20}} \put(30,20){\line(1,-2){20}}
\put(0,0){\line(3,2){30}} \put(30,0){\line(-3,2){30}} \put(0,20){\line(1,0){30}}

\newcommand{\lb}[1]{\tiny $#1$}
\put(-27,20){\lb{1}} \put(-27,-20){\lb{2}} \put(-7,-2){\lb{3}}
\put(34,-2){\lb{4}}
\put(55,20){\lb{5}}\put(55,-20){\lb{6}}\put(-7,24){\lb{3'}}\put(34,24){\lb{4'}}
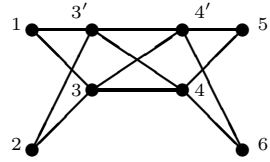
\put(-40,-40){Fig. 8. $\mbox{def}(G^{(1,1,2,2,1,1)})=0$}
\end{picture}

\end{array}
$$
\end{example}

\vspace{1.5cm}

The theorem of Berge is equivalent to the following classical result of Tutte
describing perfect matchings \cite{matching-theory}.

\begin{theorem}{\rm (Tutte; see
\cite[Theorem~2.2.1]{diestel})}\label{tutte-theorem} 
A graph $G$ has a perfect matching if and only if 
$c_0(G\setminus S) \leq | S | $ for all $S\subset  V(G)$.
\end{theorem}

We give the following characterization of perfect matchings in terms
of duplications of edges. 

\begin{corollary}\label{pepe-vila} Let $G$ be a graph. 
Then $G$ has a perfect matching if
and only if $G^f$ has a perfect matching for every edge $f$ of $G$.
\end{corollary}

\begin{proof} Assume that $G$ has a perfect matching. Let
$f_1,\ldots,f_{n/2}$ be a set of 
edges of $G$ that form a perfect matching of $V(G)$, 
where $n$ is the number of vertices of $G$. If $f=\{x_i,x_j\}$ is any
edge of $G$ and $y_i$, $y_j$ are the duplications of the vertices
$x_i$ and $x_j$ respectively, then clearly
$f_1,\ldots,f_{n/2},\{y_i,y_j\}$ form a perfect matching of $V(G^f)$.
Conversely, if $G^f$ has a perfect matching for all $f\in E(G)$, then
${\rm def}(G^f)=0$ for all $f\in E(G)$. Hence, by 
Theorem~\ref{pepe-vila-berge}, we get that ${\rm def}(G)=0$, so $G$
has a perfect matching. 
\end{proof}

The following lemma will play an important role in the proof of the
main theorem. It uses the preceding combinatorial results about
matchings to prove an algebraic equality.

\begin{lemma}\label{jan29-2011} Let $I$ be the edge ideal of a graph $G$. 
Then $(I^{k+1}\colon I)=I^k$ for $k\geq 1$.
\end{lemma}

\begin{proof} Let $F=\{f_1,\ldots,f_q\}$
be the set of all monomials $x_ix_j$ such that $\{x_i,x_j\}\in E(G)$. Given
$c=(c_i)\in\mathbb{N}^q$, we set $f^c=f_1^{c_1}\cdots f_q^{c_q}$. It
is well known that the colon ideal of two monomial ideals is a
monomial ideal, see for instance \cite[p.~137]{monalg}. In particular
$(I^{k+1}\colon I)$ is a monomial 
ideal. Clearly $I^k\subset(I^{k+1}\colon I)$. To show the
reverse inclusion it suffices to show that any monomial of
$(I^{k+1}\colon I)$ is in $I^k$. Take $x^a\in(I^{k+1}\colon I)$. Then $f_ix^a\in I^{k+1}$ for
$i=1,\ldots,q$. We may assume that $f_ix^a\notin I^{k+2}$, otherwise
$x^a\in I^k$ as required. Thus $x^{a+e_i+e_j}\in I^{k+1}\setminus
I^{k+2}$ for any $e_i+e_j$ such that
$\{x_i,x_j\}\in E(G)$. Hence, by Lemma~\ref{feb9-2011}(b),
$\nu(G^{a+e_i+e_j})=k+1$ for any $\{x_i,x_j\}\in E(G)$, that is, 
$(G^a)^{\{x_i,x_j\}}$ has a maximum matching of size $k+1$ for any edge
$\{x_i,x_j\}$ of $G$. With the notation used in the proof of 
Lemma~\ref{multiset-perfect-matching}, for any edge $\{x_i^{k_i},x_j^{k_j}\}$ of $G^a$ we have
$$
(G^a)^{\{x_i^{k_i},x_j^{k_j}\}}=(G^a)^{\{x_i,x_j\}},
$$
see Lemma~\ref{feb9-2011}(c). Then, $(G^a)^f$ has a maximum matching
of size $k+1$ for any edge $f$ 
of $G^a$. As a consequence 
$$
{\rm def}((G^a)^f)=(|a|+2)-2(k+1)=|a|-2k
$$
for any edge $f$ of $G^a$. Therefore, by
Theorem~\ref{pepe-vila-berge}, ${\rm def}(G^a)=|a|-2k$. Using
Lemma~\ref{feb9-2011}(a), we can write $x^a=x^\delta f^c$, 
where $|\delta|={\rm def}(G^a)$ and $|c|=\nu(G^a)$. Taking degrees in
the equality $x^a=x^\delta f^c$ gives
$|a|=|\delta|+2|c|=(|a|-2k)+2|c|$, that is, $|c|=k$. Then $x^a\in I^k$
and the proof is complete.
\end{proof}

\begin{proposition}\label{jan29-2011-1}
Let $I=I(G)$ be the edge ideal of a graph $G$ and let
$\mathfrak{m}=(x_1,\ldots,x_n)$. If $\mathfrak{m}\in {\rm
Ass}(R/I^k)$, then $\mathfrak{m}\in {\rm
Ass}(R/I^{k+1})$.
\end{proposition}

\begin{proof} As $\mathfrak{m}$ is an associated prime of $R/I^{k}$,
there is $x^a\notin I^k$ such that $\mathfrak{m}x^a\subset I^k$. 
By Lemma~\ref{jan29-2011} there is an edge $\{x_i,x_j\}$ of $G$ such
that $x_ix_jx^a\notin I^{k+1}$. Then, $x_\ell(x_ix_jx^a)\in I^{k+1}$
for $\ell=1,\ldots,n$, that is, $\mathfrak{m}$ is an associated prime
of $R/I^{k+1}$.
\end{proof}

To generalize from the maximal ideal to arbitrary associated primes,
we will use localization. Since this process frequently results in
disjoint graphs, we first recall the following fact about associated
primes.

\begin{lemma}\label{disjoint}{\rm (\cite[Lemma~3.4]{HaM}, see also
  \cite[Lemma~2.1]{AJ})} Let $I$ be a square-free monomial ideal in $S = K[x_1,\dots,x_m,
  x_{m+1}, \dots, x_r]$ such that 
$I=I_1S + I_2S$, where $I_1 \subset S_1 = K[x_1, \dots, x_m]$
and $I_2 \subset S_2 = K[x_{m+1}, \dots, x_r]$. Then ${\mathfrak p}\in
{\rm Ass} (S/I^k)$ 
if and only if 
${\mathfrak p}={\mathfrak p}_1S + {\mathfrak p}_2S$, where ${\mathfrak
  p}_1 \in {\rm Ass} (S_1/I_1^{k_1})$ and ${\mathfrak p}_2 \in 
{\rm Ass} (S_2/I_2^{k_2})$ with $(k_1-1) + (k_2-1) = k-1$.
\end{lemma}

Note that this lemma easily generalizes to an ideal $I=(I_1,\ldots,I_s)$ where the 
$I_i$ are square-free monomial ideals in disjoint sets of variables. 
Then ${\mathfrak p} \in {\rm Ass}(R/I^k)$ if and only if 
${\mathfrak p}=({\mathfrak p}_1,\ldots , {\mathfrak p}_s)$, where 
${\mathfrak p}_i \in {\rm Ass}(R/I_i^{k_i})$ with $(k_1-1)+\cdots 
+(k_s-1)=k-1$. 

Note that although $\mathfrak{p}_i$ is an ideal of $R$, the
generators of $\mathfrak{p}_i$ will generate a prime ideal in
any ring that contains those variables. We will abuse notation in the sequel by
denoting the ideal generated by the generators of $\mathfrak{p}_i$ in
any other ring by $ \mathfrak{p}_i$ as
well.

\medskip

We come to the main algebraic result of this paper.

\begin{theorem}\label{persistence-edge-ideals} 
Let $G$ be a graph and let $I=I(G)$ be its edge
ideal. Then 
$${\rm Ass}(R/I^k) \subset{\rm Ass}(R/I^{k+1})$$
for all $k$. That
is, the sets of associated primes of the powers of $I$ form an
ascending chain.  
\end{theorem}

\begin{proof} Recall that we are assuming that $G$ has no isolated vertices. 
Let $\mathfrak{p}$ be an associated prime of 
$R/I^k$ and let $\mathfrak{m}=(x_1,\ldots,x_n)$ be the irrelevant maximal ideal
of $R$. For simplicity of notation we may assume that
$\mathfrak{p}=(x_1,\ldots,x_r)$. Then, the set $C=\{x_1,\ldots,x_r\}$ is a vertex cover 
of $G$. By Proposition~\ref{jan29-2011-1}, we may assume that
$\mathfrak{p}\subsetneq\mathfrak{m}$. Write 
$I_{\mathfrak p}=(I_2,I_1)_\mathfrak{p}$, where $I_2$ is the ideal of
$R$ generated by all square-free monomials of degree two $x_ix_j$ whose image,
under the canonical map $R\rightarrow R_\mathfrak{p}$, 
is a minimal generator of $I_{\mathfrak p}$, and $I_1$ is the prime ideal
of $R$ generated by all variables $x_i$ whose image is a minimal generator of
$I_{\mathfrak p}$, which correspond to the isolated vertices
of the graph associated to $I_{\mathfrak p}$. The minimal generators
of $I_2$ and $I_1$ lie in $S=K[x_1,\ldots,x_r]$, and the two sets of
variables occurring in
the minimal generating sets of $I_1$ and $I_2$ (respectively) are
disjoint and their union 
is $C=\{x_1,\ldots,x_r\}$. If $I_2=(0)$, then
${\mathfrak p}$ is a minimal prime of $I$ so it is an associated
prime of $R/I^{k+1}$. Thus, we may assume $I_2\neq(0)$. An important fact is that
localization 
preserves associated primes, that is $\mathfrak{p}\in{\rm Ass}(R/I^k)$ if and only if
$\mathfrak{p}R_\mathfrak{p}\in{\rm
Ass}(R_\mathfrak{p}/(I_\mathfrak{p}R_\mathfrak{p})^k)$, see
\cite[p.~38]{Mats}. Hence, 
$\mathfrak{p}$ is in ${\rm Ass}(R/I^k)$ if and only if $\mathfrak{p}$
is in ${\rm Ass}(R/(I_1,I_2)^k)$ if and only if $\mathfrak{p}$ is 
in ${\rm Ass}(S/(I_1,I_2)^k)$. By Proposition~\ref{jan29-2011-1} and
Lemma~\ref{disjoint}, $\mathfrak{p}$ is an associated prime of
$S/(I_1,I_2)^{k+1}$. Hence, we can argue backwards to conclude that
$\mathfrak{p}$ is an associated prime of $R/I^{k+1}$. 
\end{proof}

\begin{remark} Using Proposition~\ref{jan29-2011-1} and
Lemma~\ref{disjoint}, this result can also be shown by induction 
on the number of variables because 
localizing at $\mathfrak{p}\subsetneq\mathfrak{m}$ yields the ideal
$(I_1,I_2)$ in a polynomial ring with fewer than $n$ variables. Using
induction may be useful to extend
Theorem~\ref{persistence-edge-ideals} to other classes of monomial 
ideals (for instance to edge ideals of clutters, see \cite{symboli}). In the case where
the conclusion of Lemma~\ref{jan29-2011} holds for a class of
square-free monomial ideals, this result immediately extends.  
\end{remark}

\begin{corollary}\label{square-free-chain}
Let $I$ be a square-free monomial ideal and suppose $(I^{k+1} \colon
I)=I^k$ for $k \geq 1$. Then the sets of associated primes of the
powers of $I$ form an ascending chain.  
\end{corollary}

\begin{proof}
As in Proposition~\ref{jan29-2011-1}, we first show that $\mf \in
{\rm Ass}(R/I^k)$ implies $\mf \in {\rm Ass}(R/I^{k+1})$. Assume $\mf
\in {\rm Ass}(R/I^k)$. Then there is a monomial $x^a \not\in I^k$
with $x_ix^a \in I^{k+1}$ for all $i$. By the hypothesis, $x^a
\not\in (I^{k+1} \colon I)$, so there is a square-free monomial
generator $e$ of $I$ (which can be viewed as the edge of a clutter
associated to $I$) with $ex^a \not\in I^{k+1}$. But
$x_iex^a=e(x_ix^a) \in I^{k+1}$ for all $i$, so $\mf \in {\rm
Ass}(R/I^{k+1})$.         

Recall that since $I$ is finitely generated, $(I^{k+1} \colon
I)_{\mathfrak p}= (I^{k+1}_{\mathfrak p} \colon
I_{\mathfrak p})$. Thus $(I^{k+1}_{\mathfrak p} \colon
I_{\mathfrak p})=I^k_{\mathfrak p}$.
The remainder of the argument now follows from localization, as in the
proof of Theorem~\ref{persistence-edge-ideals}, after noting that
Lemma~\ref{disjoint} applies to an arbitrary square-free monomial
ideal.   
\end{proof}

In \cite[Question~4.16]{edge-ideals} it was asked if the sets ${\rm
  Ass}(R/I^k)$ form an ascending chain for all square-free monomial
  ideals $I$. Corollary~\ref{square-free-chain} provides one possible
  approach for answering this question for some classes of square-free
  monomial ideals. However, this approach will not work for all
  square-free monomial ideals, as can be seen by the following example.

\begin{example}\label{ass-powers-ce} Let
  $R=\mathbb{Q}[x_1,\ldots,x_6]$ and let  
$I$ be the square-free monomial ideal 
$$
I=(x_1x_2x_5,\, x_1x_3x_4,\, x_1x_2x_6,\, x_1x_3x_6,\, x_1x_4x_5,\, 
x_2x_3x_4,\, x_2x_3x_5,\, x_2x_4x_6,\, x_3x_5x_6,\, x_4x_5x_6).
$$
Using {\it Normaliz\/} \cite{normaliz2} together with {\em Macaulay\/}$2$
\cite{mac2}, it is seen that $I$ is a non-normal ideal such that
$(I^2 : I)=I$ and $(I^3 : I)\neq I^2$. Nevertheless, it is not hard to see that the
sets of associated primes of the powers of  $I$ form an ascending
chain and that the index of stability of $I$ is equal to $3$.  
\end{example}

It is also of interest to note that
for square-free monomial ideals, knowing that the sets ${\rm
  Ass}(R/I^k)$ form an ascending chain immediately implies that the
sets ${\rm Ass}(I^{k-1}/I^{k})$ form an ascending chain as well. Thus
we get the following corollary of
Theorem~\ref{persistence-edge-ideals}. A similar corollary would
follow from Corollary~\ref{square-free-chain} as well. 

\begin{corollary} Let $I=I(G)$ be the edge ideal of a graph $G$, then 
${\rm Ass}(I^{k-1}/I^{k})$ form an ascending chain for $k\geq 1$. 
\end{corollary}

\begin{proof} It follows from \cite[Lemma~4.4 ]{edge-ideals} and
Theorem~\ref{persistence-edge-ideals}. 
\end{proof}

\section{Integral Closures and Stable Sets}

As mentioned in the Introduction, the sets of associated primes of the
integral closures of the powers of $I$ are also known to form an
ascending chain that stabilizes. In order to compare the stable sets
of the two chains ${\rm Ass}(R/I^k)$ and  ${\rm
  Ass}(R/\overline{I^k})$, we recall the following definition and
lemma.  

\begin{definition}\rm Let $I=(x^{v_1},\ldots,x^{v_q})$ be a monomial
ideal of $R$. The {\it Rees algebra\/} of $I$, denoted by $R[It]$, is
the monomial subring
$$
R[It]=R[x^{v_1}t,\ldots,x^{v_q}t]\subset R[t].
$$
The ring ${\mathcal F}(I)=R[It]/{\mathfrak
m}R[It]$ is called the 
{\it special fiber ring\/} of $I$. The Krull dimension of ${\mathcal
F}(I)$, denoted by $\ell(I)$, is called the {\it analytic spread\/} 
of $I$.
\end{definition}

\begin{lemma}{\cite[Proposition~7.1.17,
Exercise~7.4.10]{monalg}}\label{analytic-spread-formula} 
Let $I=(x^{v_1},\ldots,x^{v_q})$ be a monomial ideal and let $A$ be
the matrix with column vectors $v_1,\ldots,v_q$. If $\deg(x^{v_i})=d$ for all $i$, then 
$$
\mathcal{F}(I)\simeq K[x^{v_1}t,\ldots,x^{v_q}t]\simeq K[x^{v_1},\ldots,x^{v_q}]\ 
\mbox{ and }\ \ell(I)=\dim\,
K[x^{v_1},\ldots,x^{v_q}]={\rm rank}(A).
$$ 
\end{lemma}

Once again, localization will allow us to reduce to the case of the
maximal ideal. To that end, we prove a result that characterizes when
$\mathfrak{m}$ is in the stable sets of  ${\rm Ass}(R/I^k)$ and  ${\rm
  Ass}(R/\overline{I^k})$ . 

\begin{proposition}\label{feb1-2011} Let $G$ be a graph. The
following are equivalent\/{\rm :}
\begin{itemize}
\item[(a)] $\mathfrak{m}\in {\rm Ass}(R/I(G)^k)$ for some $k$.
\item[(b)] The connected components of $G$ are non-bipartite.
\item[(c)] $\mathfrak{m}\in{\rm Ass}(R/\overline{I(G)^{t}})$ for some $t$.
\item[(d)] ${\rm rank}(A)=n$, where $A$ is the incidence matrix of $G$
and $n=|V(G)|$.
\end{itemize}
\end{proposition}

\begin{proof} The equivalence between (c) and (d) follows from
\cite[Theorem~3]{mcadam} because the analytic spread of $I$ is equal
to the rank of $A$, see Lemma~\ref{analytic-spread-formula}. The
equivalence between (b) and (d) follows from the fact that 
${\rm rank}(A)=|V(G)|$ if $G$ is a connected non-bipartite graph and 
${\rm rank}(A)=|V(G)|-1$ if $G$ is a connected bipartite graph, see
\cite[Lemma~8.3.2]{monalg}. 

Let $G_1,\ldots,G_r$ be the connected components of $G$. We set
$S_i=K[V(G_i)]$ and $\mathfrak{m}_i=(V(G_i))$. Assume that
$\mathfrak{m}=(\mathfrak{m}_1,\ldots,\mathfrak{m}_r)$ is an associated
prime of $R/I(G)^k$ for some $k$. Then, by Lemma~\ref{disjoint}, 
there are positive integers $k_i$ such that $\mathfrak{m}_i$ is an
associated prime of $S_i/I(G_i)^{k_i}$. Therefore $G_i$ is
non-bipartite for all $i$ because edge ideals of bipartite graphs are
normally torsion-free \cite{ITG}. This proves that (a) implies (b).
Finally we prove that (b) implies (a). 
Assume that $G_i$ is non-bipartite for all $i$. Then, by
\cite[Corollary~3.4]{AJ}, $\mathfrak{m}_i\in {\rm
Ass}(S_i/I(G_i)^{k_i})$ for $k_i\gg 0$. Then,  
again by Lemma~\ref{disjoint}, it follows that $\mathfrak{m}$ is an
associated prime of $R/I(G)^k$ for some $k$.
\end{proof}

Combining Lemma~\ref{analytic-spread-formula} with
Proposition~\ref{feb1-2011}(d) illustrates the importance of the
analytic spread in determining associated primes. Localizing will
allow the use of these results for primes other than $\mathfrak{m}$,
but this will require control over the analytic spread of
the edge ideal of a disconnected graph.

\begin{lemma}\label{analytic-spread-additive} Let $L_1,L_2$ be
monomial ideals with disjoint sets of 
variables. If $L_1,L_2$ are gene\-rated by monomials of degrees $d_1$
and $d_2$ respectively, then $\ell(L_1+L_2)=\ell(L_1)+\ell(L_2)$.
\end{lemma}

\begin{proof} We set $L=L_1+L_2$. Let $g_1,\ldots,g_r$ and
$h_1,\ldots,h_s$ be the minimal generating sets of $L_1$ and $L_2$
  respectively that consist of monomials. By hypothesis $L_1$ (resp.
$L_2$) lives in a polynomial ring $K[\mathbf{x}]$ (resp.
$K[\mathbf{y}])$, where $\mathbf{x}=\{x_1,\ldots,x_q\}$ and
$\mathbf{y}=\{y_1,\ldots,y_m\}$. We set $R=K[\mathbf{x},\mathbf{y}]$.
The special fiber ring of $L$ can be written
as
$$
\mathcal{F}(L)\simeq
K[\mathbf{x},\mathbf{y},u_1,\ldots,u_r,z_1,\ldots,z_s]/(\mathbf{x},\mathbf{y},J),
$$
where $J$ is the presentation ideal of the Rees algebra $R[Lt]$ and
$u_1,\ldots,u_r,z_1,\ldots,z_s$ is a set of new indeterminates. The
ideal $J$ is the kernel of the map
$$
K[\mathbf{x},\mathbf{y},u_1,\ldots,u_r,z_1,\ldots,z_s]\rightarrow
R[Lt],\ \ \ \ \ x_i\mapsto x_i,\, y_j\mapsto y_j,\, u_i\mapsto g_it,\, 
z_j\mapsto h_jt.
$$
Since $J$ is a toric ideal, there is a generating set of $J$
consisting of binomials of the form
$$
x^{\alpha}y^{\beta}u^{\gamma}z^\delta-x^{\alpha'}y^{\beta'}u^{\gamma'}z^{\delta'}
$$
such that $x^{\alpha}y^{\beta}g^{\gamma}h^{\delta}
t^{|\gamma|+|\delta|}=
x^{\alpha'}y^{\beta'}g^{\gamma'}h^{\delta'}t^{|\gamma'|+|\delta'|}$.
From this equation we get $x^\alpha g^\gamma=x^{\alpha'}g^{\gamma'}$, $y^\beta
h^\delta=y^{\beta'}h^{\delta'}$ and
$t^{|\gamma|+|\delta|}=t^{|\gamma'|+|\delta'|}$. Hence
$$
|\alpha|+d_1|\gamma|=|\alpha'|+d_1|\gamma'|,\ 
|\beta|+d_2|\delta|=|\beta'|+d_2|\delta'|,\
{|\gamma|+|\delta|}={|\gamma'|+|\delta'|}.
$$ 
We claim that $\deg(x^{\alpha}y^{\beta})=0$ if and only if
$\deg(x^{\alpha'}y^{\beta'})=0$. Assume that
$\deg(x^{\alpha}y^{\beta})=0$, i.e., $\alpha=\beta=0$. From the first
equality we have $|\alpha'|=d_1(|\gamma|-|\gamma'|)$. From the second
and third equality we get
$$
|\beta'|+d_2|\delta'|=d_2|\delta|=d_2({|\gamma'|+|\delta'|}-|\gamma|)\
\Rightarrow\ |\beta'|=d_2(|\gamma'|-|\gamma|).
$$
As $|\alpha'|\geq 0$ and $|\beta'|\geq 0$, we get $\gamma-\gamma'=0$.
Thus $\alpha'=\beta'=0$. This proves the claim. Therefore one has the
following simpler expression for the special fiber ring of $L$
\begin{equation}\label{feb13-11}
\mathcal{F}(L)\simeq 
K[u_1,\ldots,u_r,z_1,\ldots,z_s]/P\simeq
K[g_1t,\ldots,g_rt,h_1t,\ldots, h_st],
\end{equation}
where $P$ is the toric ideal of $K[g_1t,\ldots,g_rt,h_1t,\ldots,
h_st]$. Let $A_1$ (resp. $A_2$) be the matrix whose columns are the
exponent vectors of the monomials $g_1t,\ldots,g_rt$ (resp.
$h_1t,\ldots,h_st$) and let $A$ be the matrix whose columns are the
exponent vectors of $g_1t,\ldots,g_rt,h_1t,\ldots,h_st$. The sets of
variables $\mathbf{x}$ and $\mathbf{y}$ are disjoint. Therefore 
${\rm rank}(A)={\rm rank}(A_1)+{\rm rank}(A_2)$. Since 
$$
\mathcal{F}(L_1)=K[g_1t,\ldots,g_rt]\ \mbox{ and }\ 
\mathcal{F}(L_2)=K[h_1t,\ldots,h_st],
$$ 
using Lemma~\ref{analytic-spread-formula} and Eq.~(\ref{feb13-11}), it follows that 
$\ell(L)=\ell(L_1)+\ell(L_2)$. 
\end{proof}

\begin{remark} When $L_2$ is generated by a set of variables, which is
the case that we really need, the
lemma follows at once from \cite[Corollary~6.2, p.~43]{McAdam} because in this case the
set of variables form an asymptotic sequence over $L_1$ in the sense of
\cite{McAdam}.
\end{remark}

The sets ${\rm Ass}(R/I^i)$ and ${\rm Ass}(R/\overline{I^i})$
stabilize for large $i$. The next result shows that, for edge ideals, their
corresponding stable sets are equal. 

\begin{theorem}\label{ass=assic} Let $I$ be the edge ideal of a graph
  $G$. There exists 
a positive integer $N$ such that ${\rm Ass}(R/I^k)={\rm
Ass}(R/\overline{I^k})$ for
$k\geq N$.
\end{theorem}

\begin{proof} Recall that we are assuming that $G$ has no isolated vertices. By
\cite{brod}
there is a positive integer $N_1$ such that ${\rm Ass}(R/I^{N_1})={\rm
  Ass}(R/I^k)$ for $k \geq N_1$, and by \cite{ratliff,ratliff-increasing}, 
there is a positive integer $N_2$ such that ${\rm
Ass}(R/\overline{I^{N_2}})={\rm Ass}(R/\overline{I^{k}})$ 
for $k\geq N_2$. Let $N={\rm max}\{N_1,N_2\}$, and assume that $k\geq
N$. First we show the 
inclusion ``$\subset$''. Take $\mathfrak{p}$ in
${\rm Ass}(R/I^k)$. 

Case (I): $\mathfrak{p}=\mathfrak{m}$. By 
Proposition~\ref{feb1-2011}, $\mathfrak{p}\in{\rm
Ass}(R/\overline{I^i})$ for some $i$. Hence $\mathfrak{p}\in{\rm
Ass}(R/\overline{I^k})$ because the sets ${\rm
Ass}(R/\overline{I^j})$ form an ascending sequence, see
\cite[Proposition~3.4, p.~13]{McAdam}.

Case (II): $\mathfrak{p}=(x_1,\ldots,x_r)\subsetneq\mathfrak{m}$. Let
$I_1$, $I_2$, and $S$ be as in the proof of
Theorem~\ref{persistence-edge-ideals} and let $X_i$ be the set of
variables that occur in the minimal generating set of $I_i$. Notice
that $\mathfrak{p}=(X_1,X_2)$. As $\mathfrak{p}$ is an associated
prime of $S/(I_1+I_2)^{k}$, applying Lemma~\ref{disjoint} to
$I_1S+I_2S$, where we regard $I_i$  
as an ideal of $S_i=K[X_i]$, we can write ${\mathfrak p}={\mathfrak
p}_1S + {\mathfrak p}_2S$, where ${\mathfrak
  p}_1 \in {\rm Ass} (S_1/I_1^{k_1})$ and ${\mathfrak p}_2 \in 
{\rm Ass} (S_2/I_2^{k_2})$, with $(k_1-1) + (k_2-1) = k-1$. Notice 
that $\mathfrak{p}_i=(X_i)$. Thus, applying
Proposition~\ref{feb1-2011} to the graph $G_2$ associated to
$I_2$, we get that the rank of the incidence matrix $A_{G_2}$ of
$G_2$ is $|X_2|$. On the other hand $\ell(I_2)$, the analytic spread
of $I_2$, is equal to the Krull dimension of the edge subring
$K[G_2]$, 
which is equal to the rank of 
$A_{G_2}$ (see Lemma~\ref{analytic-spread-formula}). Since
$I_1$ is generated by $|X_1|$ variables, one has 
$\ell(I_1)=|X_1|$. By Lemma~\ref{analytic-spread-additive}, the
analytic spread $\ell(I_1+I_2)$ is equal to 
$|X_1|+|X_2|={\rm ht}(\mathfrak{p})$. Thus, using
\cite[Theorem~3]{mcadam}, we conclude that
$\mathfrak{p}\in S/\overline{(I_1+I_2)^i}$ for $i\gg 0$. Then, 
$\mathfrak{p}R_\mathfrak{p}\in
R_\mathfrak{p}/\overline{(I_1+I_2)_\mathfrak{p}^i}=
R_\mathfrak{p}/\overline{I_\mathfrak{p}^i}$. Consequently, by
\cite[Corollary, p.~38]{Mats} and the fact that the 
integral closure of ideals commute with localizations, we get
$\mathfrak{p}\in R/\overline{I^k}$. 

The inclusion ``$\supset$'' holds for any ideal $I$ of a commutative
Noetherian ring $R$ by a result of Ratliff
\cite{ratliff,ratliff-increasing}, see \cite[Proposition~3.17]{McAdam}
for additional details.  
\end{proof}

\begin{corollary}\label{ntf-ass-corollary} Let $G$ be a graph and let $I$ be its edge ideal.
Then $I$ is normally torsion-free if and only if 
${\rm Ass}(R/\overline{I^i})={\rm Ass}(R/I)$ for $i\geq 1$.  
\end{corollary}

\begin{proof} $\Rightarrow$) This implication follows at once by
noticing that $I$ is normal, i.e., $\overline{I^i}=I^i$ for $i\geq 1$. 

$\Leftarrow$) Since ${\rm Ass}(R/I)\subset{\rm
Ass}(R/{I^i})$ for $i\geq 1$, it suffices to show that 
${\rm Ass}(R/{I^i})\subset{\rm Ass}(R/I)$ for $i\geq 1$. 
Let $\mathfrak{p}$ 
be an associated prime of $R/{I^i}$ and let $N$ be the index
of stability of $I$. Then, by Theorem~\ref{persistence-edge-ideals},
$\mathfrak{p}$ is an associated prime of $R/I^{N}$. Hence, by
Theorem~\ref{ass=assic}, $\mathfrak{p}$ is an associated prime of
$R/\overline{I^k}$ for $k\gg 0$. Thus by hypothesis $\mathfrak{p}$ is
an associated prime of $I$.
\end{proof}

\begin{procedure}\label{mac-norm-proc} The following simple procedure for {\it
Macaulay\/}$2$ (version 1.4) decides
whether ${\rm Ass}(R/I^3)$ is contained in ${\rm Ass}(R/I^4)$ and
whether we have the equality ${\rm Ass}(R/I^3)={\rm Ass}(R/I^4)$. It
also computes $\overline{I^4}$ and 
decides whether ${\rm Ass}(R/I^4)$ is equal to ${\rm
Ass}(R/\overline{I^4})$.
\begin{verbatim}
R=QQ[x1,x2,x3,x4,x5,x6,x7,x8,x9];
load "normaliz.m2";
I=monomialIdeal(x1*x2,x2*x3,x1*x3,x3*x4,x4*x5,x5*x6,x6*x7,x7*x8,x8*x9,x5*x9);
isSubset(ass(I^3),ass(I^4))
ass(I^3)==ass(I^4)
(intCl4,normRees4)=intclMonIdeal I^4;
intCl4'=substitute(intCl4,R); 
ass(monomialIdeal(intCl4'))==ass(I^4)
\end{verbatim}
\end{procedure}

The next example was computed using version 1.4 of {\em
Macaulay\/}$2$ \cite{mac2}. This version allows the use of {\it Normaliz\/}
\cite{normaliz2} inside {\em Macaulay\/}$2$ in order
to compute the integral closure of a monomial ideal and the normalization of the Rees
algebra of a monomial ideal. Example~\ref{intcl1.m2} shows that
although the stable sets of ${\rm Ass}(R/{I}^{i})$ and ${\rm
Ass}(R/\overline{I^{i}})$ are equal, they do not
need to be reached at the same power. 

\begin{example}\label{intcl1.m2} Let $R=\mathbb{Q}[x_1,\ldots,x_9]$ and let 
$I=I(G)$ be the edge ideal of the graph below (Fig. 9). Notice that
this example was computed without using Theorem~\ref{ass=assic}.
$$
\setlength{\unitlength}{.04cm} \thicklines
\begin{picture}(0,40)(40,-5)
\put(0,0){\circle*{4.2}} \put(30,0){\circle*{4.2}} \put(60,0){\circle*{4.2}}
\put(-20,20){\circle*{4.2}} \put(-20,-20){\circle*{4.2}}
\put(80,20){\circle*{4.2}}\put(80,-20){\circle*{4.2}}
\put(110,20){\circle*{4.2}}\put(110,-20){\circle*{4.2}}

\put(0,0){\line(-1,-1){20}}\put(0,0){\line(-1,1){20}}\put(0,0){\line(1,0){60}}
\put(60,0){\line(1,1){20}}
\put(60,0){\line(1,-1){20}}\put(-20,-20){\line(0,1){40}}
\put(80,20){\line(1,0){30}}\put(80,-20){\line(1,0){30}}
\put(110,-20){\line(0,1){40}}

\newcommand{\lb}[1]{\tiny $#1$}
\put(-27,20){\lb{1}} \put(-27,-20){\lb{2}} \put(-7,-2){\lb{3}}\put(30,3){\lb{4}}
\put(64,-2){\lb{5}} \put(73,20){\lb{6}}\put(115,20){\lb{7}}\put(115,-20){\lb{8}}
\put(73,-21){\lb{9}}
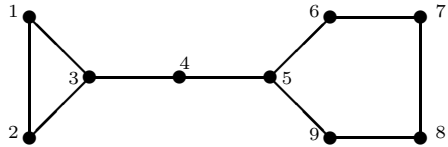
\put(-35,-40){Fig. 9. Graph $G$ with non-normal $I(G)$}
\end{picture}
$$

\vspace{1.5cm}

Using {\em Macaulay\/}$2$ (see Procedure~\ref{mac-norm-proc}), 
together with the fact that the index of stability of $I$ is at most
$8$ \cite[Corollary~4.3]{AJ} and 
the fact that the stable set of ${\rm Ass}(R/\overline{I^i})$ is
contained in the stable set 
of ${\rm Ass}(R/{I^i})$ \cite[Proposition~3.17]{McAdam}, we get
\begin{eqnarray*}
I^i=\overline{I^i},\ i=1,2,3,\
\overline{I^4}=I^4+(x^a),\, \overline{I^5}=I^5+I(x^a), \mbox{ where }
x^a=x_1x_2x_3x_5x_6x_7x_8x_9,&&\\ 
{\rm Ass}(R/I^i)={\rm Ass}(R/\overline{I^i})\ \mbox{for }\ i\neq 4\
\mbox{ and }\ {\rm Ass}(R/\overline{I^4})\subsetneq{\rm Ass}(R/I^4),& &\\ 
{\rm Ass}(R/I)\subsetneq{\rm Ass}(R/I^2)\subsetneq{\rm
Ass}(R/I^3)\subsetneq{\rm Ass}(R/I^4)={\rm
Ass}(R/I^i)\ \mbox{for }\ i\geq 4,& &\\
{\rm
Ass}(R/\overline{I^3})\subsetneq{\rm
Ass}(R/\overline{I^4})\subsetneq{\rm Ass}(R/\overline{I^5})={\rm
Ass}(R/\overline{I^i})\ \mbox{ for }\ i\geq 5.& &
\end{eqnarray*}
\end{example}

\medskip

\begin{center}
ACKNOWLEDGMENT
\end{center}
\noindent The authors would like to thank an anonymous referee for
providing us  
with useful comments and suggestions.

\end{document}